\documentclass[11pt]{amsart}
\usepackage{geometry}

\usepackage{tikz-cd}
\usepackage[all]{xy}
\usepackage{amsthm,amsmath,amssymb,amsfonts}
\usepackage{mathrsfs}
\usepackage{bbm}

\theoremstyle{plain}

\newtheorem{theorem}{Theorem}
\newtheorem{lemma}[theorem]{Lemma}

\newtheorem{proposition}[theorem]{Proposition}
\newtheorem{corollary}[theorem]{Corollary}

\numberwithin{theorem}{section}
\numberwithin{equation}{theorem}

\theoremstyle{definition}
\newtheorem{definition}[theorem]{Definition}

\newtheorem{example}[theorem]{Example}
\newtheorem{remark}[theorem]{Remark}
\newtheorem{question}[theorem]{Question}

\newtheorem*{question*}{Question}

\def\Hom{\mathrm{Hom}}

\makeatletter

\begin{document}

\title{The weak bialgebra structures on $\Bbbk^{\oplus n}$}
\author{Jingheng Zhou}

\address{Jingheng Zhou: School of Mathematics, 
Shanghai University of Finance and Economics, Shanghai 200433, 
Shanghai, China}

\email{zhoujingheng@mail.shufe.edu.cn}

\maketitle

{\noindent{\bf Abstract:} Weak bialgebras are natural generalizations of bialgebras that play a fundamental role in quantum groups, tensor categories, and representation theory. In this paper, we systematically classify all (weak) bialgebra structures on the  finite‑dimensional semisimple algebra $\Bbbk^{\oplus n}$, with a distinguished basis of primitive orthogonal idempotents. For the bialgebra case, we prove that such structures are in bijection with finite monoids, and that bialgebra isomorphisms correspond exactly to monoid isomorphisms. Extending the result to weak bialgebras, we show that every weak bialgebra structure uniquely determines a finite small category.
 }
 \\
\noindent{\bf Key Words:} weak bialgebra, quiver, monoid, small category \\
\noindent{\bf MSC (2020):} 16T10, 16G20, 16T05, 18M05\\

\section{Introduction}

Bialgebras and Hopf algebras are fundamental algebraic structures situated at the intersection of algebra, geometry, and mathematical physics. They provide a natural framework for the study of quantum groups, representation theory, and tensor categories. A classical bialgebra over a field $\Bbbk$ is both an associative algebra and a coassociative coalgebra, subject to the compatibility conditions that the comultiplication and counit are algebra morphisms and that the unit is a coalgebra morphism.

A natural generalization, known as a \emph{weak bialgebra}, was introduced by Böhm, Nill, and Szlachányi in their foundational work on weak Hopf algebras\cite{BNS1999}. In a weak bialgebra, the compatibility conditions between the algebra and coalgebra structures are weakened: the coproduct is multiplicative but need not preserve the unit, and the counit satisfies a weaker set of axioms. Schauenburg pointed out that these objects are $\times_R$-bialgebras, and conversely, that any $\times_R$-bialgebra with a separable basis $R$ is a weak bialgebra\cite{Schauenburg2001}. Indeed, it was shown by Etingof, Nikshych, and Ostrik that every multi-fusion category is equivalent to the category of modules over a regular semisimple weak Hopf algebra\cite{ENO2005}. The theory of weak bialgebras has since developed into a rich area of research, with connections to monoidal categories, comodule algebras and so on \cite{BCJ2011,WWW2022}.

A fundamental method for understanding these structures in concrete terms is through the path algebra of a quiver, whose representation theory and algebraic properties are well understood. The path space \(\Bbbk Q\) of a quiver \(Q\) admits a natural algebra structure via path combination. Dually, Chin and Montgomery introduced a natural coalgebra structure on the same underlying space, the path coalgebra, whose comultiplication splits paths at any vertex\cite{CM1997}. Unfortunately, those two structures are not compatible in general. They do not form a bialgebra in the usual sense. One classical weak bialgebra structure on 
$\Bbbk Q$ is given by the point-wise tensor product, i.e. $\Delta(p)=p\otimes p$ for every path $p$, but this is not the only one.
This leads naturally to the following fundamental question:
\begin{question}\label{xxque1.1}
     Can we classify all weak bialgebra structures on $\Bbbk Q$ for a given 
     quiver $Q$?
 \end{question}
 
Significant progress towards this question has been made from two complementary directions. On the one hand, Cibils and Rosso introduced Hopf quivers, which yield Hopf algebras whose underlying coalgebras are path coalgebras\cite{CR2002}; this idea was later extended to the weak setting by Ahmed and Li\cite{AL2011}. 
Quiver methods have since been widely used in the study of pointed Hopf algebras and tensor categories\cite{HY2016, Crawford2023, MT2019}. 
On the other hand, Herschend characterised when a partitioning map $\Delta:\Bbbk Q\to \Bbbk Q\otimes \Bbbk Q$ is an algebra map\cite{He2008}. 
Huang and Torrecillas verified that the path coalgebra of an arbitrary quiver admits natural bialgebra structures \cite{HT2013}. 
More recently, Zhang and Zhou classified all weak bialgebra structures on $\Bbbk Q$, where $Q$ is an acyclic quiver with two vertices\cite{ZZ2022}. 
Despite these advances, a complete classification for general quivers remains open, and the present work aims to contribute to this endeavour.

In this paper, we study weak bialgebra structures on the commutative algebra $B = \Bbbk^{\oplus n}$, the direct sum of $n$ copies of the base field $\Bbbk$, with the standard pointwise multiplication. 
In fact, $B$ can be regarded as the path algebra for a quiver with $n$ vertices and no arrows.
This algebra has a distinguished basis $\{e_1,\dots,e_n\}$ of primitive orthogonal idempotents. For bialgebras, it is classical that a bialgebra structure on $\Bbbk^{\oplus n}$ corresponds precisely to a monoid structure on the index set $\{1,\dots,n\}$: the comultiplication is given by
\[
\Delta(e_k) = \sum_{i\cdot j = k} e_i \otimes e_j,
\]
where $\cdot$ is the monoid multiplication, and the counit is $\varepsilon(e_k) = \delta_{1k}$ for the identity element $1$. In particular, the number of isomorphism classes of bialgebra structures on $\Bbbk^{\oplus n}$ equals the number of isomorphism classes of monoids with $n$ elements.

For weak bialgebras, the situation is more complicated. Let $m = \varepsilon(1_B)$, and $m$ may not be $1$ under this situation, so the set of idempotents with counit $1$ forms a subset of size $m$, and the remaining idempotents have counit $0$. 
As a consequence of the general theory, weak bialgebra structures on $\Bbbk^{\oplus n}$ correspond bijectively to finite categories having exactly $m$ objects and $n$ morphisms, where the $m$ counital idempotents correspond to the identity morphisms and the remaining $n-m$ basis elements correspond to non-identity morphisms. 
The coproduct encodes the decomposition of a morphism into composable pairs.

The paper is organized as follows. In Section 2, we recall some basic definitions of bialgebras and weak bialgebras. In Section 3, we prove the correspondence between bialgebra structures on $\Bbbk^{\oplus n}$ and monoids. In Section 4, we establish the correspondence between weak bialgebra structures and finite categories. In Section 5, we give the complete classification for $n=3$, including the Grothendieck rings of the bialgebra cases and the explicit weak bialgebra structures. 

Throughout, we work over an arbitrary base field $\Bbbk$. We refer the reader to \cite{Montgomery1993,Sweedler1969} for general information about coalgebras and bialgebras.

\section{Preliminaries}
In this section, we will recall some basic definitions.
\begin{definition}\cite[Section 3.1]{Sweedler1969}
\label{xxdef2.1}Let $(B, m, 1 )$ be an algebra and $(B, \Delta, \varepsilon)$ be a coalgebra. If one of the following conditions 
are satisfied, then $(B,m,1,\Delta,\varepsilon)$ is called a bialgebra and denoted simply $B$.
\begin{enumerate}
\item[(a)]
$m, 1$ are coalgebra maps.
\item[(b)]
$\Delta, \varepsilon$ are algebra maps.
\end{enumerate}
\end{definition}

If $B$ is a bialgebra, then we have 
\[\Delta(1)=1\otimes 1 \text{ and } \Delta(gh)=\Delta(g)\Delta(h).\]
If we remove the condition $\Delta(1)=1\otimes 1$, then $\Delta$ is no longer an algebra map. Thus, we call $\Delta$ the prealgebra map. A prealgebra map is an algebra map if and only if $\Delta(1)=1\otimes 1$.

The next is about the definition of a weak bialgebra.

\begin{definition} \cite[Definition 2.1]{BNS1999}
\label{def2.2}
A {\it weak bialgebra} is a vector space $B$ over the
base field $\Bbbk$ with the structures of
\begin{enumerate}
\item[(a)]
an associative algebra $(B, m, 1 )$ with multiplication
$m: B\otimes B\to B$ and unit $1 \in B$, and
\item[(b)]
a coassociative coalgebra $(B, \Delta, \varepsilon)$ with
comultiplication $\Delta: B\to B\otimes B$ and couint
$\varepsilon: B\to \Bbbk$
\end{enumerate}
satisfying the following conditions.
\begin{enumerate}
\item[(i)]
The comultiplication $\Delta: B\to B\otimes B$ is a
prealgebra morphism.
\item[(ii)]
The unit and counit satisfy
\begin{equation*}
(\Delta(1 )\otimes 1 )(1 \otimes \Delta(1 ))
=(\Delta\otimes Id) \Delta(1 )
=(1 \otimes \Delta(1 ))(\Delta(1 )\otimes 1 )
\end{equation*}
and
\begin{equation*}
\varepsilon( xyz)=\sum\varepsilon(x y_{(1)})\varepsilon(y_{(2)}z)=
\sum\varepsilon(x y_{(2)})\varepsilon(y_{(1)}z),
\end{equation*}
where $\Delta(y)=\displaystyle\sum y_{(1)}\otimes y_{(2)}$ is the Sweedler notation.
\end{enumerate}
\end{definition}
We refer to \cite{BCJ2011,BNS1999,NTV2003,NV2002} for many other basic definitions related to weak bialgebras
and weak Hopf algebras.

The next is about quiver, see \cite{ASS2006}.
\begin{definition}[\cite{ASS2006}]
A quiver $Q = (Q_0, Q_1, s, t)$ is a quadruple consisting
of two sets: $Q_0$ (whose elements are called points, or vertices) and $Q_1$
(whose elements are called arrows), and two maps $s, t : Q_1 \to Q_0$ which
associate to each arrow $\alpha \in Q_1$ its source $s(\alpha) \in Q_0$ and its target $t(\alpha) \in Q_0$, respectively.  
\end{definition}

A path of length
$l \geq 1$ with source $a_0$ and target $a_l$ (or, more briefly, from $a_0$ to $a_l$) is a sequence
\[(a_0|\alpha_1,\alpha_2, \cdots,\alpha_l|a_l),\]
satisfying $s(\alpha_1)=a_0,t(\alpha_i)=s(\alpha_{i+1})=a_i$ for $i=1,\cdots,l-1$ and $t(\alpha_l)=a_l$.
For each point $a \in Q_0$, a path of length 0 is called a trivial path at $a$, denoted by $e_a$. Let $\Bbbk Q$ be the linear space spanned by all paths of $Q$, then we can define the product of two paths by the composition of two paths (if they can not be composed, the product is define to 0). This is the well-known path algebra. The dual coalgebra is defined as follows. For any path $(a_0|\alpha_1,\alpha_2, \cdots,\alpha_l|a_l)$,
\[\Delta(p)=e_{a_0}\otimes p+\sum_{i=1}^{l-1} 
(a_0|\alpha_1,\alpha_2, \cdots,\alpha_i|a_i)\otimes (a_i|\alpha_{i+1},\alpha_{i+2}, \cdots,\alpha_l|a_l)+p\otimes e_{a_l},\]
and $\varepsilon(p)=\delta_{0l}$ where $l$ is the length of $p$. 

To address Question\ref{xxque1.1}, we first consider some simple examples. For example, can we classify all weak bialgebra structures on $\Bbbk^{\oplus n}$ or when $Q$ only has two vertices? In 2022, Zhang and Zhou answered the question when $Q$ is an acyclic quiver and only has two vertices\cite{ZZ2022}.

\begin{proposition}\cite[Lemma 7.5, Proposition 7.7]{ZZ2022}\label{xxlem2.3}
Suppose $Q$ is the quiver with two vertices $\{1,2\}$ and
$w$ arrows from $1$ to $2$. Let $A$ be the
path algebra $\Bbbk Q$. Then there are 5 types of weak bialgebra
structures on $A$ up to equivalences.
\begin{enumerate}
\item[(a)]
$\Delta(e_1)=e_1\otimes e_1, \Delta(e_2)=e_2\otimes e_2+e_1\otimes e_2
+e_2\otimes e_1$, $\varepsilon(e_1)=1, \varepsilon(e_2)=0$, and for any
arrow $r$ from 1 to 2, $\Delta(r)= e_1\otimes r+r\otimes e_1$ and
$\varepsilon(r)=0$.
\item[(b)]
$\Delta(e_1)=e_1\otimes e_1+e_2\otimes e_2, \Delta(e_2)=e_2\otimes e_1
+e_1\otimes e_2$, $\varepsilon(e_1)=1, \varepsilon(e_2)=0$,
and for any arrow $r$ from 1 to 2, $\Delta(r)=r\otimes e_1
+e_1\otimes r$ and $\varepsilon(r)=0$.
\item[(c)]
$\Delta(e_2)=e_2\otimes e_2, \Delta(e_1)=e_1\otimes e_1+e_1\otimes e_2
+e_2\otimes e_1$, $\varepsilon(e_2)=1, \varepsilon(e_1)=0$, and for any
arrow $r$ from 1 to 2, $\Delta(r)= e_2\otimes r+r\otimes e_2$ and
$\varepsilon(r)=0$.
\item[(d)]
$\Delta(e_2)=e_1\otimes e_1+e_2\otimes e_2, \Delta(e_1)=e_2\otimes e_1
+e_1\otimes e_2$, $\varepsilon(e_2)=1, \varepsilon(e_1)=0$,
and for any arrow $r$ from 1 to 2, $\Delta(r)=r\otimes e_2
+e_2\otimes r$ and $\varepsilon(r)=0$.
\item[(e)]
$\Delta(e_i)=e_i\otimes e_i$, $\varepsilon(e_i)=1$ for $i=1,2$,
and the Jacobson radical $J$ of $A$ is a subcoalgebra of $A$.
\end{enumerate}
\end{proposition}

\section{The bialgebra structures on $\Bbbk^{\oplus n}$}

In this section, we will explore the bialgebra structures on $\Bbbk^{\oplus n}$.
When $n=2$, by Proposition \ref{xxlem2.3}, if
there are no arrows form vertex 1 to 2, then there are only two bialgebra structures up to equivalent on 
$\Bbbk^{\oplus 2}$. 

For the usual case, let $B=\Bbbk^{\oplus n}=\Bbbk e_1\oplus\Bbbk e_2\oplus\cdots\oplus\Bbbk e_n$, and fix the algebra structure on $B$ where $\{e_1,\cdots,e_n\}$ are primitive orthogonal idempotents in $B$,
then we will find all the coalgebra structures on the bialgebra $B$.

The first lemma is about the counit of the primitive orthogonal idempotents $e_1,\cdots,e_n$.
\begin{lemma}\label{xxlem3.1}
One of $\varepsilon(e_1),\cdots,\varepsilon(e_n)$ is equal to one while others are zero.
\end{lemma}
\begin{proof}
    For each $i\in$ $\{1,2,\cdots,n\}$, since 
    $\varepsilon(e_i)=\varepsilon(e_i^2)=\varepsilon(e_i)^2$,
    then $\varepsilon(e_i)=1$ or $0$.
    
    Notice that the unit $1=e_1+e_2+\cdots+e_n$ in $B$, and $\varepsilon(1)=1$, then
    \[\varepsilon(1)=\varepsilon(e_1)+\varepsilon(e_2)+\cdots+\varepsilon(e_n)=1.\]
    Hence exactly one of the $\varepsilon(e_i)$ is equal to $1$ while others are $0$.
\end{proof}

By symmetry of $e_1, e_2,\cdots, e_n$, we may assume that
$\varepsilon(e_1)=1$, $\varepsilon(e_2)=\cdots=\varepsilon(e_n)=0$. 
Below we will determine the coalgebra structures on $B$ under this assumption.

Let \[
\Delta(e_k)=\sum_{i,j=1}^n a_{ij}^k e_i\otimes e_j,
\]
the following lemmas are about the coefficients $\{a_{ij}^k\}$.

\begin{lemma}\label{xxlem3.2}
    For each $a_{ij}^k$, $a_{ij}^k=0$ or $1$.
\end{lemma}
\begin{proof}
    Since $\Delta(e_k)=\Delta(e_k^2)=(\Delta(e_k))^2$,
    then we have  
    \begin{align*}
    \sum_{i,j=1}^n a_{ij}^k e_i\otimes e_j &=\Delta(e_k)
    =(\Delta(e_k))^2
    =\left(\sum_{i,j=1}^n a_{ij}^k e_i\otimes e_j\right)^2\\
    &=\sum_{i,j=1}^n (a_{ij}^k)^2 e_i^2\otimes e_j^2
    =\sum_{i,j=1}^n (a_{ij}^k)^2 e_i\otimes e_j.
    \end{align*}

    Thus, $a_{ij}^k=(a_{ij}^k)^2$, i.e. $a_{ij}^k=1$ or $0$.   
\end{proof}

\begin{lemma}\label{xxlem3.3}
    For each pair $(i,j)$, there is exactly one $k\in\{1,2,\cdots,n\}$ such that the coefficient $a_{ij}^k=1$, and all other coefficients are $0$.
\end{lemma}
\begin{proof}
Notice that $\Delta(1)=1\otimes 1=\displaystyle\sum_{i,j=1}^n e_i\otimes e_j$, and $\Delta(1)=\Delta(e_1)+\Delta(e_2)+\cdots+\Delta(e_3)$, then each $e_i\otimes e_j$ exactly appears in one $\Delta (e_k)$. 
\end{proof}

\begin{lemma}\label{xxlem3.4}
    For each $i,k$, $a_{1i}^k=a_{i1}^k=\delta_{ik}$.
\end{lemma}
\begin{proof}
    Since $(\varepsilon\otimes \mathrm{id})\Delta(e_k)=(\mathrm{id}\otimes\varepsilon)\Delta(e_k)=e_k$, 
    we will have the terms $e_i\otimes e_1$ and $e_1\otimes e_i$ only appear in $\Delta(e_i)$.
\end{proof}

By the lemma \ref{xxlem3.3}, for every pair $(i,j)$, we can find a unique $k$ corresponding to this pair. Thus, we define an operator "$\cdot$" on the set $S=\{1,2,\cdots, n\}$ as follows:
\[\forall i,j\in S, \text{ let } i\cdot j=k \text{ if } a_{ij}^k=1. \]

Lemma \ref{xxlem3.4} means that for any $i\in S$,
\[i\cdot 1=1\cdot i=i,\]
which means $1$ is the identity in $(S,\cdot)$.

\begin{lemma}\label{xxlem3.5}
    For all $k,p,q,j\in\{1,2,\cdots,n\}$, we have 
    \begin{align*}\tag{3-1}\label{eq3-1}
    \sum_{i=1}^n  a_{pq}^i a_{ij}^k = \sum_{i=1}^n a_{pi}^k a_{qj}^i.
    \end{align*}
\end{lemma}
\begin{proof}
    From the coassociativity $(\Delta\otimes \mathrm{id})\Delta(e_k)=(\mathrm{id}\otimes\Delta)\Delta(e_k)$, we have
    \begin{align*}
    \text{LHS}&=\sum_{i,j=1}^n a_{ij}^k \Delta(e_i)\otimes e_j
    =\sum_{i,j=1}^n a_{ij}^k \left(\sum_{p,q=1}^n a_{pq}^i e_p\otimes e_q\right)\otimes e_j\\
    &=\sum_{p,q,j=1}^n \left(\sum_{i=1}^n a_{ij}^k a_{pq}^i\right) e_p\otimes e_q\otimes e_j,\\
    \text{RHS}&=\sum_{p,i=1}^n a_{pi}^k e_p\otimes \Delta(e_i)
    =\sum_{j,i=1}^n a_{pi}^k e_p\otimes \left(\sum_{q,j=1}^n a_{qj}^i e_q\otimes e_j\right)\\
    &=\sum_{j,p,q=1}^n \left(\sum_{i=1}^n a_{pi}^k a_{qj}^i\right) e_p\otimes e_q\otimes e_j.
    \end{align*}
    Comparing coefficients yields:
    \begin{equation*}
    \sum_{i=1}^n  a_{pq}^i a_{ij}^k = \sum_{i=1}^n a_{pi}^k a_{qj}^i.
    \end{equation*}
\end{proof}
By Lemma \ref{xxlem3.3}, for every pair $(i,j)$, we can find a unique $k$ corresponding to this pair. Thus, we define a
 binary operator "$\cdot$" on the set $S=\{1,2,\cdots, n\}$ as follows:
\[\forall i,j\in S, \text{ let } i\cdot j=k \text{ if } a_{ij}^k=1. \]

Lemma \ref{xxlem3.4} means that for any $i\in S$,
\[i\cdot 1=1\cdot i=i,\]
which means $1$ is the identity in $(S,\cdot)$. Meanwhile, 
Lemma \ref{xxlem3.5} tells us that 
the operator "$\cdot$" on $S$ is associative, thus $(S,\cdot)$ is a monoid and we have the following theorem.

\begin{theorem}\label{xxthm3.6}
    Let $S=\{1,2,\cdots,n\}$, and $B=\Bbbk^{\oplus n}$ be a commutative algebra over $\Bbbk$ with primitive orthogonal idempotents $\{e_1,e_2,\cdots, e_n\}$. 
    Then $S$ is a monoid if and only if 
    $B$ is a bialgebra.
\end{theorem}
\begin{proof}
When $B$ is a bialgebra, we have proved that $S$ has a monoid structure from $B$.

Conversely, if $S$ is a monoid, without loss of generality, let $1$ be the unit. Then for any $e_i(i=2,3,\cdots, n)$, define the counit $\varepsilon$ and the comultiplications $\Delta$ as follows:
\[
    \varepsilon(e_i)=\delta_{1i}, \quad 
    \Delta(e_i)=\sum_{k\cdot l=i} e_k\otimes e_l.
\]
It's straightforward to verify that 
\begin{align*}
    &(\varepsilon\otimes id)\Delta=id=(id\otimes \varepsilon)\Delta,\\
    &(\Delta\otimes id)\Delta=(id \otimes\Delta)\Delta,
\end{align*}
and both $\Delta$ and $\varepsilon$ are algebra maps. Therefore, $B$ is a bialgebra.
\end{proof}

The next theorem tells us the number of bialgebra 
structures on $B$ is equal to the number of monoids with n elements.

\begin{theorem}\label{xxthm3.7}
    Let $B_1$ and $B_2$ be two bialgebra structures on $B$, and $S_1$, $S_2$
    be the corresponding monoid. Then 
    $B_1$ is isomorphic to $B_2$ as bialgebra if and only if $S_1$ is isomorphic to $S_2$.
\end{theorem}
\begin{proof}
    If $B_1$ is isomorphic to $B_2$, let $f:B_1\to B_2$ be the isomorphism, then for any $i\in S$, $f(e_i)$
    still is a primitive orthogonal idempotent, and let it be $e_{\sigma(i)}$. 

    We claim that $\sigma$ is a monoid isomorphism between $S_1$ and $S_2$. Firstly, $\sigma$ is a bijection. If $f(e_i)=f(e_j)$, then 
    $f(e_i)=f(e_i)f(e_j)=f(e_i e_j)=\delta_{ij}f(e_i)$. Therefore $i=j$, and $\sigma$ is a bijection since $S$ is finite.

    Secondly, since $f$ is a bialgebraic isomorphism,
    $\Delta_2 \circ f=(f\otimes f)\circ \Delta_1 $. For $i,j\in S_1$, let $k=i\cdot j$, which means 
    $e_i\otimes e_j$ appears in $\Delta_1(e_k)$. Thus $f(e_i)\otimes f(e_j)$ appears in $\Delta_2(f(e_k))$,
    i.e. $e_{\sigma(i)}\otimes e_{\sigma(j)}$ appears in $\Delta_2(e_{\sigma(k)})$. Therefore
    \[\sigma(i)\sigma(j)=\sigma(k)=\sigma(i\cdot j).\]

    Lastly, if $i_0$ is the identity in $S_1$,
    then $\varepsilon(e_{i_0})=1$. Thus, 
    \[\varepsilon_2(e_{\sigma(i_0)})=\varepsilon_2(f(e_{i_0}))=\varepsilon_1(e_{i_0})=1,\]
    which means $\sigma(i_0)$ is the identity in $S_2$.

    As a result, $\sigma$ is a monoid isomorphism.

    Conversely, if $\sigma$ is a monoid isomorphism, we can also construct a bialgebraic isomorphism from $B_1$ to $B_2$.
\end{proof}

\begin{example}
    When $n=2$, there are two possible monoid structures on $S=\{1,2\}$ if $1$
    is the identity. Thus, up to isomorphic,
    there are two different bialgebra structures on $\Bbbk^{\oplus 2}.$
\end{example}

\section{The weak bialgebra structures on $\Bbbk^{\oplus n}$}

In this section, we turn our attention to weak bialgebra structures. Similarly, for the primitive orthogonal idempotents, let \[
\Delta(e_k)=\sum_{i,j=1}^n a_{ij}^k e_i\otimes e_j.
\]

Similar to Lemma \ref{xxlem3.1} and Lemma \ref{xxlem3.2}, we have the following lemmas.

\begin{lemma}\label{xxlem4.1}
    For each $a_{ij}^k$, $a_{ij}^k=0$ or $1$.
\end{lemma}
\begin{proof}
    Since $\Delta$ is a prealgebra morphism,
    the proof is similar to Lemma \ref{xxlem3.2}.
\end{proof}

\begin{lemma}\label{xxlem4.2}
    For each $e_k$, $\varepsilon(e_k)=0$ or $1$.
\end{lemma}
\begin{proof}
    Since $\varepsilon(xyz)=\sum\varepsilon(x y_{(1)})\varepsilon(y_{(2)}z)=
    \sum\varepsilon(x y_{(2)})\varepsilon(y_{(1)}z),$
    then
    \[\varepsilon(e_k)=\varepsilon(e_k^3)
    =\sum_{i,j=1}^n a_{ij}^k\varepsilon(e_k e_i)\varepsilon(e_j e_k)=a_{kk}^k \varepsilon(e_k)^2.\]
    By Lemma \ref{xxlem4.1}, $a_{kk}^k=0$ or $1$, thus $\varepsilon(e_k)=0$ or $1$.
\end{proof}

By Lemma \ref{xxlem4.2}, there exists a positive integer $m$ such that $\varepsilon(1)=m$. In fact, the dimensions of source counital subalgebra $B_s$ and target counital subalgebra $B_t$ both are $m$. 
Without loss of generality, let 
\[\varepsilon(e_1)=\cdots=\varepsilon(e_m)=1,\]while $\varepsilon(e_{m+1})=\cdots=\varepsilon(e_n)=0$.

\begin{proposition}\label{xxpro4.3}
   Let $V=span\{e_{m+1},\cdots, e_n\}$, then $V$ is a coideal of $B$.
\end{proposition}
\begin{proof}
 We verify the two coideal conditions:

\begin{enumerate}
\item It is direct that \(\varepsilon(V)=0\).
\item We show \(\Delta(V)\subseteq V\otimes B+B\otimes V\).
\end{enumerate}

Take any \(k>m\) (so \(e_k\in V\)). Write
\[
\Delta(e_k)=\sum_{p,q=1}^n a_{pq}^k\, e_p\otimes e_q,
\]
where each \(a_{pq}^k\in\{0,1\}\). Suppose, for contradiction, that some term \(e_p\otimes e_q\) with \(p,q\le m\) appears in this sum, i.e. \(a_{pq}^k=1\).

Apply the counital axiom \((\varepsilon\otimes \mathrm{id})\Delta(e_k)=e_k\). On the left‑hand side, the term \(a_{pq}^k e_p\otimes e_q\) contributes
\[
a_{pq}^k\,\varepsilon(e_p)\,e_q = 1\cdot 1\cdot e_q = e_q,
\]
and since \(q\le m\), this gives a nonzero component in \(W=\operatorname{span}\{e_1,\dots,e_m\}\). Other terms may contribute further, but in any case the left‑hand side is a linear combination of basis vectors that includes \(e_q\) with a positive coefficient (because all coefficients are non‑negative and no cancellation can occur). However, the right‑hand side is \(e_k\), which lies in \(V\) and hence has no component in \(W\). This contradiction shows that no such term can exist.

Therefore, for every \(k>m\), the coproduct \(\Delta(e_k)\) contains no component in \(W\otimes W\). Hence each term of \(\Delta(e_k)\) has at least one factor outside \(W\), i.e. at least one factor in \(V\). Consequently,
\[
\Delta(e_k)\in V\otimes B+B\otimes V.
\]
Since the \(e_k\) with \(k>m\) form a basis of \(V\), we obtain
\[
\Delta(V)\subseteq V\otimes B+B\otimes V.
\]

Thus \(V\) is a coideal.
\end{proof}
\begin{corollary}
    $C=B/V$ is a coalgebra with each $\overline{e_i} (1\leq i\leq m)$ is a group-like element.
\end{corollary}
\begin{proof}
    Let \(\pi:B\to B/V\) be the canonical projection and $\overline{x}=\pi(x)$. Set
\[
\overline{\Delta}(\overline{x})=(\pi\otimes\pi)\Delta(x),\qquad 
\overline{\varepsilon}(\overline{x})=\varepsilon(x).
\]
Since $V$ is coideal, then $(C,\overline{\Delta},\overline{\varepsilon})$ is a coalgebra. For $1\leq k\leq m$,  $\overline{\varepsilon}(\overline{e_k})=1$, and 
\[\overline{\Delta}(\overline{e_k})=\sum_{i,j=1}^m a_{ij}^k \overline{e_i} \otimes \overline{e_j}.\]
Moreover, since each $a_{ij}^k$ equals $0$ or $1$, by the counital axiom, 
\[\overline{\Delta}(\overline{e_k})=\overline{e_k} \otimes \overline{e_k},\]
i.e. $\overline{e_k}$ is a group-like element in $C$.
\end{proof}

\begin{corollary}\label{xxcor4.5}
    If $V=0$, then $B$ is a weak bialgebra with each primitive orthogonal idempotent is a group-like element.
\end{corollary}

We now establish the precise passage from a weak bialgebra structure on \(B = k^{\oplus n}\) to a small category \(\mathcal{C}\). Let \(\{e_1, \dots, e_n\}\) be the basis of primitive orthogonal idempotents with \(e_i e_j = \delta_{ij} e_i\), and put \(m := \#\{i \mid \varepsilon(e_i)=1\}\). By Lemma 4.2, after relabelling we may assume
\[
\varepsilon(e_1)=\cdots=\varepsilon(e_m)=1,\qquad \varepsilon(e_{m+1})=\cdots=\varepsilon(e_n)=0.
\]

The category \(\mathcal{C}\) is reconstructed by the following steps.

 The objects of \(\mathcal{C}\) are the indices that are distinguished by the counit:
\[
\mathrm{Ob}(\mathcal{C}) := \{1,2,\dots,m\}.
\]
For each object \(i\), the identity morphism is defined to be the basis idempotent \(\mathrm{id}_i := e_i\).

 The set of all morphisms is identified with the whole basis \(\{e_1,\dots,e_n\}\). For a non-identity basis element \(e_k\) with \(k>m\), its source \(s(e_k)\) and target \(t(e_k)\) are uniquely determined by the counital axioms. Indeed, from
\[
(\varepsilon \otimes \mathrm{id})\Delta(e_k)=e_k
\]
we see that there is a unique \(s(e_k)\in \{1,\dots,m\}\) such that the coefficient of \(e_{s(e_k)} \otimes e_k\) in \(\Delta(e_k)\) equals \(1\). Similarly, from
\[
(\mathrm{id} \otimes \varepsilon)\Delta(e_k)=e_k
\]
there is a unique \(t(e_k)\in \{1,\dots,m\}\) such that the coefficient of \(e_k \otimes e_{t(e_k)}\) in \(\Delta(e_k)\) equals \(1\). Thus \(e_k\) is a morphism from \(s(e_k)\) to \(t(e_k)\), and we write
\[
e_k \in \mathrm{Hom}_{\mathcal{C}}\bigl(s(e_k),\, t(e_k)\bigr).
\]
For \(k\le m\), the element \(e_k\) is precisely the identity morphism at the object \(k\), i.e. it acts as a two-sided unit for composition.

 Let \(x\) and \(y\) be two morphisms such that \(t(x)=s(y)\). Their composition \(y\circ x\) is encoded by the coproduct: if \(y\otimes x\) appears as a summand in \(\Delta(z)\) for some basis element \(z\), then we define
\[
y\circ x := z.
\]
If no such summand exists, the composition is defined to be zero. 

 The identity laws are immediate consequences of the counit property. For \(i\le m\), the condition \((\varepsilon \otimes \mathrm{id})\Delta(e_i)=e_i\) forces \(\Delta(e_i)=e_i\otimes e_i\), hence
\[
e_i \circ e_i = e_i.
\]
Moreover, for any morphism \(e_k\), the defining properties of \(s(e_k)\) and \(t(e_k)\) imply that \(e_{t(e_k)}\circ e_k = e_k\) and \(e_k \circ e_{s(e_k)}=e_k\), so the identity morphisms behave as required.

Associativity of composition follows directly from the coassociativity of \(\Delta\). Indeed, the identity
\[
(\Delta \otimes \mathrm{id})\Delta(e_k) = (\mathrm{id} \otimes \Delta)\Delta(e_k)
\]
means that the two ways of inserting an intermediate object in the decomposition of \(e_k\) coincide. In terms of morphisms, this gives exactly
\[
(f\circ g)\circ h = f\circ (g\circ h)
\]
for any three composable morphisms \(f,g,h\). Thus all axioms of a small category are satisfied.

\begin{theorem}\label{thm:2}
There is a bijective correspondence between weak bialgebra structures on \(k^{\oplus n}\) (with the fixed pointwise multiplication) and small categories \(\mathcal{C}\) having exactly \(m = \#\{i \mid \varepsilon(e_i)=1\}\) objects and \(n\) morphisms, where the coproduct \(\Delta\) gives the decomposition of a morphism into composite pairs. Conversely, given any finite category \(\mathcal{C}\) with \(n\) morphisms, its category algebra \(k\mathcal{C}\) inherits a weak bialgebra structure by taking the morphisms as a basis, defining multiplication by composition (or zero if not composable), defining the coproduct by
\[
\Delta(f) = \sum_{\substack{g,h:\\ g\circ h = f}} g \otimes h,
\]
and defining the counit by \(\varepsilon(f)=1\) if \(f\) is an identity morphism and \(0\) otherwise.
\end{theorem}

\begin{proof}
We have already shown that every weak bialgebra on \(k^{\oplus n}\) gives a category. Conversely, starting from a finite category with morphism set \(\{f_1,\dots,f_n\}\), the vector space with basis \(\{f_i\}\) becomes a weak bialgebra with the structures described above. It is straightforward to verify that the weak monoidality axioms (2-1) and (2-2) are exactly the associativity and unit laws of the category. The two constructions are inverse to each other up to isomorphism.
\end{proof}

\section{Example: the weak bialgebra structures on $\Bbbk^{\oplus 3}$}
In this section, we will give an example when $n=3$. The next theorem gives all bialgebra structures on $\Bbbk^{\oplus 3}$.

\begin{theorem}\label{thm:1}
Let $B=\Bbbk^{\oplus3}=\Bbbk e_1\oplus\Bbbk e_2\oplus\Bbbk e_3$ be a bialgebra. Then $B$ is isomorphic to one of the following 7 bialgebras.
\begin{enumerate}
    \item[(1)] $\begin{aligned}[t]
    \varepsilon(e_1)&=1, \quad \Delta(e_1)=e_1\otimes e_1+e_2\otimes e_3+e_3\otimes e_2,\\
    \varepsilon(e_2)&=0,\quad \Delta(e_2)=e_1\otimes e_2+e_2\otimes e_1+e_3\otimes e_3,\\
    \varepsilon(e_3)&=0, \quad\Delta(e_3)=e_1\otimes e_3+e_3\otimes e_1+e_2\otimes e_2.
    \end{aligned}$
    \item[(2)] $\begin{aligned}[t]
    \varepsilon(e_1)&=1, \quad \Delta(e_1)=e_1\otimes e_1+e_2\otimes e_2,\\
    \varepsilon(e_2)&=0,\quad \Delta(e_2)=e_1\otimes e_2+e_2\otimes e_1,\\
    \varepsilon(e_3)&=0, \quad\Delta(e_3)=e_1\otimes e_3+e_3\otimes e_1+e_2\otimes e_3+e_3\otimes e_2+e_3\otimes e_3.
    \end{aligned}$
    \item[(3)]
    $\begin{aligned}[t]
    \varepsilon(e_1)&=1, \quad \Delta(e_1)=e_1\otimes e_1,\\
    \varepsilon(e_2)&=0,\quad \Delta(e_2)=e_1\otimes e_2+e_2\otimes e_1+e_2\otimes e_2+e_2\otimes e_3,\\
    \varepsilon(e_3)&=0, \quad\Delta(e_3)=e_1\otimes e_3+e_3\otimes e_1+e_3\otimes e_2+e_3\otimes e_3.
    \end{aligned}$
    \item[(4)]
    $\begin{aligned}[t]
    \varepsilon(e_1)&=1, \quad \Delta(e_1)=e_1\otimes e_1,\\
    \varepsilon(e_2)&=0,\quad \Delta(e_2)=e_1\otimes e_2+e_2\otimes e_1+e_2\otimes e_2+e_3\otimes e_2,\\
    \varepsilon(e_3)&=0, \quad\Delta(e_3)=e_1\otimes e_3+e_3\otimes e_1+e_2\otimes e_3+e_3\otimes e_3.
    \end{aligned}$
    \item[(5)]
    $\begin{aligned}[t]
    \varepsilon(e_1)&=1, \quad \Delta(e_1)=e_1\otimes e_1,\\
    \varepsilon(e_2)&=0,\quad \Delta(e_2)=e_1\otimes e_2+e_2\otimes e_1+e_2\otimes e_2+e_3\otimes e_2+e_2\otimes e_3,\\
    \varepsilon(e_3)&=0, \quad\Delta(e_3)=e_1\otimes e_3+e_3\otimes e_1+e_3\otimes e_3.
    \end{aligned}$
    \item[(6)]
    $\begin{aligned}[t]
    \varepsilon(e_1)&=1, \quad \Delta(e_1)=e_1\otimes e_1,\\
    \varepsilon(e_2)&=0,\quad \Delta(e_2)=e_1\otimes e_2+e_2\otimes e_1+e_2\otimes e_2+e_3\otimes e_2+e_2\otimes e_3+e_3\otimes e_3,\\
    \varepsilon(e_3)&=0, \quad\Delta(e_3)=e_1\otimes e_3+e_3\otimes e_1.
    \end{aligned}$
    \item[(7)]
    $\begin{aligned}[t]
    \varepsilon(e_1)&=1, \quad \Delta(e_1)=e_1\otimes e_1,\\
    \varepsilon(e_2)&=0,\quad \Delta(e_2)=e_1\otimes e_2+e_2\otimes e_1+e_2\otimes e_2+e_3\otimes e_3,\\
    \varepsilon(e_3)&=0, \quad\Delta(e_3)=e_1\otimes e_3+e_3\otimes e_1+e_3\otimes e_2+e_2\otimes e_3.
    \end{aligned}$
\end{enumerate}
\end{theorem}

\begin{proof}
    By Theorem \ref{xxthm3.7}, we only need to decide the Cayley table for the monoid $S=\{1,2,3\}$.
    Without loss of generality, let $1$ be the identity in the monoid $S$. Then by \cite{Ta1953},
    there are 7 monoids on $S$ up to isomorphic, and the corresponding bialgebra structures are listed.
\end{proof}

\begin{remark}
    \begin{enumerate}
        \item     Case 3 and Case 4 are non-cocommutative while others are cocommutative.
        \item     Except for Case 3, Case 4 and Case 6, others are co-semisimple.
    \end{enumerate}
\end{remark}

In $\mathcal{B}=\mathrm{Mod}_\Bbbk (B)$, there are three simple objects: $V_i=Be_i(i=1,2,3)$. For $i=1,\cdots,7$, 
let $K_i$ be the Grothendieck ring of the bialgebra $B$ in Theorem \ref{thm:1} case (i). 
Note that $B$ is a bialgebra, and $\dim_\Bbbk V_i=1(i=1,2,3)$, then for any $i,j=1,2,3$,
  $\dim_\Bbbk (V_i\otimes_\Bbbk V_j)=1$, thus $V_i\otimes_\Bbbk V_j$ is one of $V_1,V_2,V_3$.
Therefore, $V_i\otimes_\Bbbk V_j\cong V_k$
if and only if $a_{ij}^k=1$.
Then we have the following proposition.

\begin{proposition}\label{xxpro5.3}
Notes as above, the isomorphism classes among the $K_i$'s are exactly:
\begin{enumerate}
    \item $K_1\cong \mathbb Z[C_3]$.
    \item $K_2\cong K_7\cong \mathbb Z[C_2]\times \mathbb Z$.
    \item $K_3,\ K_4 \text{ are noncommutative and } K_3\not\cong K_4$.
    \item $K_5\cong \mathbb Z^3$.
    \item $K_6\cong \mathbb Z[t]/(t^2)\times \mathbb Z$.
\end{enumerate}
Hence among $K_1,\dots,K_7$ the only nontrivial isomorphism is $K_2\cong K_7$; all other pairs are pairwise non-isomorphic.
\end{proposition}

\begin{proof}
We can get the multiplication laws from the comultiplications in Theorem~\ref{thm:1}. For example, in case (1) we get
\[
[V_2]^2=[V_3],\quad [V_2][V_3]=[V_1],\quad [V_3]^2=[V_2],
\]
so $K_1\cong \mathbb Z[C_3]$. In case (2), $[V_2]^2=[V_1]$ and $[V_3]$ is absorbing, hence $K_2\cong \mathbb Z[C_2]\times \mathbb Z$. In case (7), the set $\{[V_2],[V_3]\}$ forms a cyclic group of order 2 with identity $[V_2]$, and $[V_1]-[V_2]$ is a complementary idempotent, so $K_7\cong \mathbb Z[C_2]\times \mathbb Z$, thus $K_2\cong K_7$.

For case (3), we have
\[
[V_2]^2=[V_2],\quad [V_3]^2=[V_3],\quad [V_2][V_3]=[V_2],\quad [V_3][V_2]=[V_3],
\]
and for case (4),
\[
[V_2]^2=[V_2],\quad [V_3]^2=[V_3],\quad [V_2][V_3]=[V_3],\quad [V_3][V_2]=[V_2].
\]
Both rings are noncommutative, so they are distinct from all commutative rings. They are not isomorphic to each other because an isomorphism would have to send the two non‑unit idempotent generators to themselves (up to permutation) and preserve the relation $ab=a$ or $ab=b$, which is impossible.

In case (5), the multiplication is commutative and the semigroup is a semilattice with an absorbing element, hence its $\mathbb Z$-algebra is $\mathbb Z^3$. In case (6), we have $[V_2]$ absorbing and $[V_3]^2=[V_2]$, so $([V_3]-[V_2])^2=0$; thus $K_6\cong \mathbb Z[t]/(t^2)\times \mathbb Z$.

Finally, to distinguish the remaining commutative rings: $K_6$ contains a nonzero nilpotent, whereas $K_2,K_5,K_7$ are reduced. Also, $K_2$ and $K_7$ have exactly two primitive idempotents (since $\mathbb Z[C_2]$ has only trivial idempotents over $\mathbb Z$), while $K_5\cong \mathbb Z^3$ has three. Therefore $K_2,K_7\not\cong K_5$.

Thus the only isomorphism is $K_2\cong K_7$, and all other pairs are non-isomorphic.
\end{proof}

As for the weak bialgebra structure, if $\varepsilon(1_B)=1$, $B$ is a bialgebra.
 If $\varepsilon(1_B)=3$, by Corollary \ref{xxcor4.5}, $B$ is a weak bialgebra with each $e_i$ is a group-like elements.

 The only case left is $\varepsilon(1_B)=2$.
 Without loss of generality, suppose $\varepsilon(e_1)=\varepsilon(e_2)=1$ and $\varepsilon(e_3)=0$.
 By Theorem \ref{thm:2}, construct a category $\mathcal{C}$ with two object $\{O_1,O_2\}$. Then $e_i(i=1,2)$ corresponds to the identity morphism in $\Hom (O_i,O_i)$. As for $e_3$, it corresponds to a morphism (still named by $e_3$) in one of the following four homomorphism spaces \[\Hom(O_1,O_1),\Hom(O_1,O_2),\Hom(O_2,O_1),\Hom(O_2,O_2).\]
 By symmetry, there are two cases: $e_3\in \Hom(O_1,O_1)$ or $e_3\in \Hom(O_1,O_2)$.

 If $e_3\in \Hom(O_1,O_2)$, then $e_2\circ e_3=e_3\circ e_1=e_3$, and in $B$, we have
 \[\Delta(e_1)=e_1\otimes e_1, \Delta(e_2)=e_2\otimes e_2, \Delta(e_3)=e_2\otimes e_3+e_3\otimes e_1.\]
 If $e_3\in \Hom(O_1,O_1)$, then $\{e_1,e_3\}$ spans a two dimensional bialgebra $B_1$ and $\{e_2\}$ spans a one-dimensional bialgebra $B_2$. Thus $B\cong B_1\oplus B_2$. By Lemma \ref{xxlem2.3}, $B_1$ has two non-isomorphic bialgebra structures.
\begin{theorem}\label{xxthm5.4}
    Let $B=\Bbbk^{\oplus3}=\Bbbk e_1\oplus\Bbbk e_2\oplus\Bbbk e_3$ be a weak bialgebra, and $B_s$ and $B_t$ be the source and target subalgebras. Then $B$ is isomorphic to one of the followings.
    \begin{enumerate}
        \item If $\dim B_s=\dim B_t=1$, then $B$ is a bialgebra, and 7 structures are listed in Theorem
        \ref{thm:1}.
        \item If $\dim B_s=\dim B_t=3$, then $e_1$, $e_2$, $e_3$ are group-like elements.
        \item If $\dim B_s=\dim B_t=2$, 
        then the 3 structures are listed above.
    \end{enumerate}
    As a result, there are 11 non-isomorphic weak bialgebra structures on $B$.
\end{theorem}

\subsection*{Acknowledgments} 
Jingheng Zhou was partially supported by NSFC (No. 12301053).

\bibliographystyle{plain}

\end{document}